\documentclass[11pt]{amsart}
\usepackage{tikz}
\usepackage[backref=page]{hyperref} 

\usepackage{graphicx}

\usepackage{float, caption} 
\newcommand{\bb}{\mathbb}

\newcommand{\C}{\bb C}
\newcommand{\h}{\bb H}
\newcommand{\Z}{\bb Z}
\newcommand{\R}{\bb R}

\newcommand{\Q}{\bb Q}

\newcommand{\cc}{\tilde{C}}

\newcommand\beann{\begin{eqnarray*}}

\newcommand\eeann{\end{eqnarray*}}
\newcommand{\vol}{\operatorname{vol}}

\newcommand{\Om}{\Omega}

\newcommand{\F}{\mathcal F}
\newcommand{\I}{\mathcal I}
\newcommand{\<}{\left\langle}
\renewcommand{\>}{\right\rangle}
\newcommand{\eps}{\epsilon}

\newcommand{\p}{\mathcal P}

\newtheorem{Theorem}{Theorem}
\numberwithin{Theorem}{section}

\newtheorem{lemma}[Theorem]{Lemma}

\newcommand{\vf}{\varphi}      
\newcommand{\foh}{\frac{1}{2}}
\newcommand{\gz}{\zeta}      
\newcommand{\mattwos}[4]
{\bigl( \begin{smallmatrix}
                        #1  & #2   \\
                        #3 &  #4
\end{smallmatrix} \bigr)
}
\newcommand{\bo}{{\bf1}}

\newtheorem*{lemma*}{Lemma}
\newtheorem*{question*}{Question}

\newtheorem*{theorem*}{Theorem}

\numberwithin{equation}{section}
\begin{document}
\title{Radial Density in Apollonian Packings}

\author{Jayadev S. Athreya, Cristian Cobeli, and Alexandru Zaharescu}


\email{jathreya@iilinois.edu}
\email{cristian.cobeli@imar.ro}

\email{zaharesc@iilinois.edu}

\address{J.A. and A.Z: Department of Mathematics, University of Illinois Urbana-Champaign, 1409 W. Green Street, Urbana, IL 61801\\  C.C. and A.Z. \textit{Simion Stoilow} Institute of Mathematics of the Romanian Academy 
21 Calea Grivitei, 
PO. Box 1-764, RO-70700, Bucharest,
Romania}

%


\begin{abstract} Given $\p$,  an \textit{Apollonian Circle Packing}, 
and a circle $C_0 = \partial B(z_0, r_0)$ in $\p$, color the set of disks in
$\p$ tangent to $C_0$ red. What proportion of the concentric circle $C_{\epsilon} = \partial B(z_0,
r_0 +
\epsilon)$ is red, and what is the behavior of this quantity as $\epsilon \rightarrow 0$? Using equidistribution of closed horocycles on the modular
surface $\h^2/SL(2, \Z)$, we
show that the answer is $\frac{3}{\pi} = 0.9549\dots$ We also describe an observation due to Alex Kontorovich connecting the rate of this convergence in the Farey-Ford packing to the Riemann Hypothesis. For the analogous
problem for Soddy
Sphere packings, we find that the limiting radial density is $\frac{\sqrt{3}}{2V_T}=0.853\dots$,
where
$V_T$ denotes the volume of an ideal hyperbolic tetrahedron with dihedral angles $\pi/3$. 
\end{abstract}
\maketitle

\section{Introduction}\label{sec:intro} Following the beautiful series of articles ~\cite{GLMWY:2003, GLMWY:2005,
GLMWY1:2006, GLMWY2:2006} by Graham, Lagarias, Mallow, Wilks, Yan,
and Sarnak's elegant letter~\cite{Sarnak:2008} to Lagarias, the study of Apollonian
Circle Packings (ACPs) has been a topic of great interest to number theorists and homogeneous
dynamicists. In particular, counting and sieving problems have been the subject of much study, and
important recent advances have been made by Bourgain, E.~Fuchs, Kontorovich, Oh, Sarnak, Shah and
many others. A
crucial tool in these recent
developments has been dynamics on infinite-volume quotients of hyperbolic $3$-space $\h^3$, 
we refer the reader to~\cite{BF:2011, KO:2011}
and to the excellent surveys~\cite{Fuchs, Kontorovich:2013, Oh1:2013, Oh2:2013} for an overview of
the literature.

We consider the following geometric problem: if we fix an Apollonian packing,  a circle in that packing,
and a small $\epsilon > 0$, what is the probability that a
randomly chosen point at distance $\epsilon$ from the fixed circle
lies inside a circle of the packing that is tangent to the fixed circle?
Intuitively, as $\epsilon$ tends to zero, it becomes more and more
probable that a randomly chosen point as above does lie
inside a circle tangent to the given one. Thus, the basic question we study is this: \textit{is it true that for
any Apollonian packing and any circle
in the packing, the probability above tends to $1$ as $\epsilon$ tends to zero?}
As we will see below, there is a universal answer to this problem, in
the sense that the above probability does have a limit as $\epsilon$
tends to zero. Furthermore, the limit is not 1, but $3/\pi\approx
0.95493$, and indeed, this limit is the same for all Apollonian
packings and all circles in the packing. To prove this result, we use homogeneous dynamics on the (finite-volume)
modular surface $\h^2/SL(2, \Z)$. Our methods generalize to give a similar universal answer to the
analagous radial density problem for Soddy Sphere Packings
(SSPs) using equidistribution of horospheres on a finite-volume hyperbolic 3-orbifold.

\subsection{Apollonian Circle Packings} Recall that an ACP is given by starting with any
configuration of $3$ mutually tangent circles $C_1, C_2, C_3 \subset \hat{\C}  = \C \cup \infty$
(where circles through $\infty$ are straight lines). Apollonius' theorem, from which these packings
take their name, states that there are exactly two circles (say $C_0$ and $C_4$) tangent to all
three. Now, take any triple of circles from the initial quintipule $\{C_0, C_1, C_2, C_3, C_4\}$ and
construct the remaining mutually tangent circle, and proceed inductively to obtain a packing. An
example of the first few circles in a packing is shown in Figure~\ref{fig:acp}, where the circles
are labeled with their \emph{bends}, that is, the inverse of their radii.


\begin{figure} [tbh]
\hfill
\begin{minipage}[b]{0.47\textwidth}
   \includegraphics[width=0.99\textwidth]{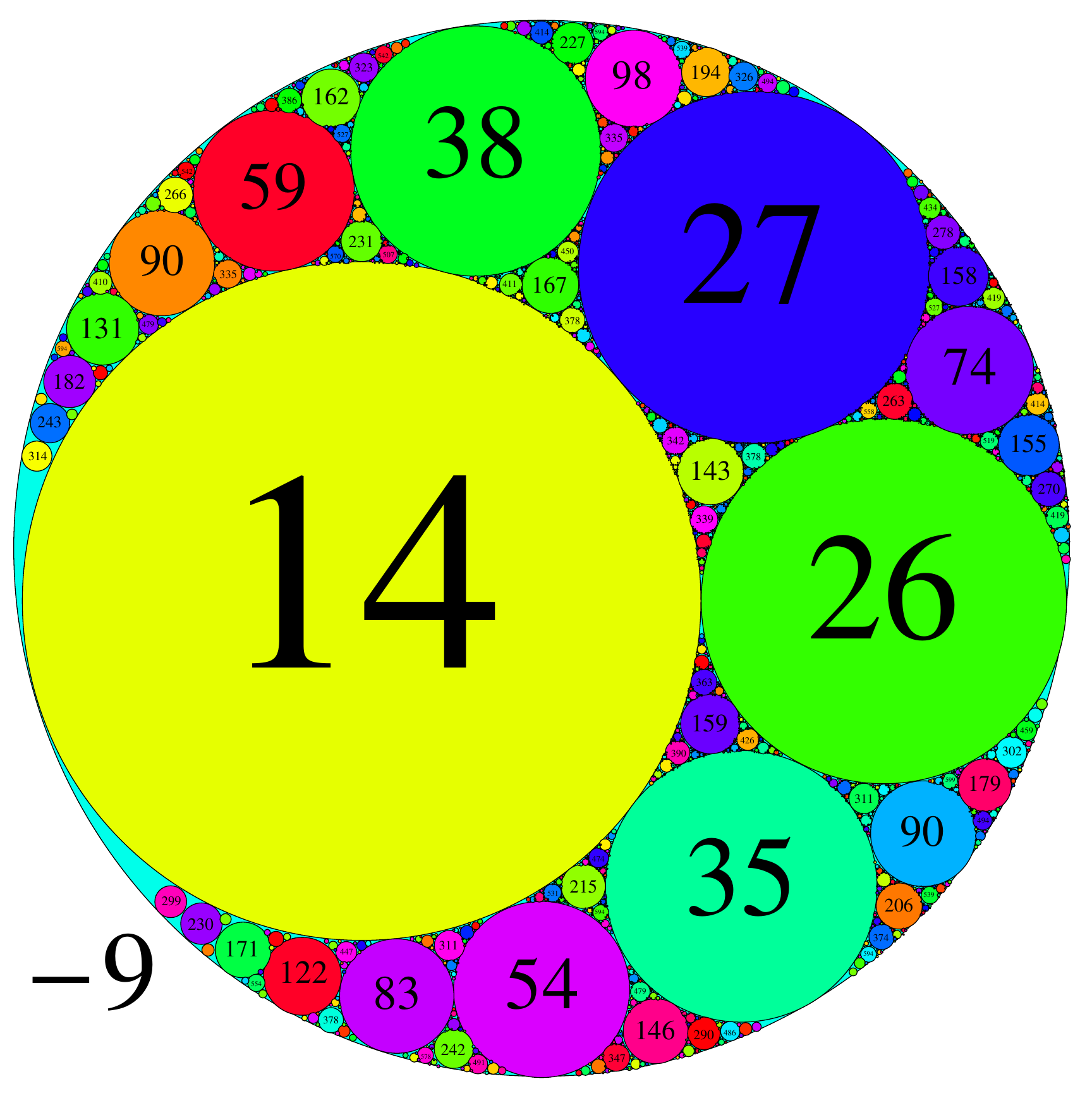}
\captionsetup{width=0.98\textwidth}
   \caption{\small Seven iterated insertions of circles of the Apollonian Circle Packing generated
by the triple of circles of \mbox{curvatures}: $[14, 26, 27]$, colored randomly.}
   \label{fig:acp}
\end{minipage}
\hfill\hfill\hfill\hfill\hfill
\begin{minipage}[b]{0.48\textwidth}
   \includegraphics[angle=-90,width=0.99\textwidth]{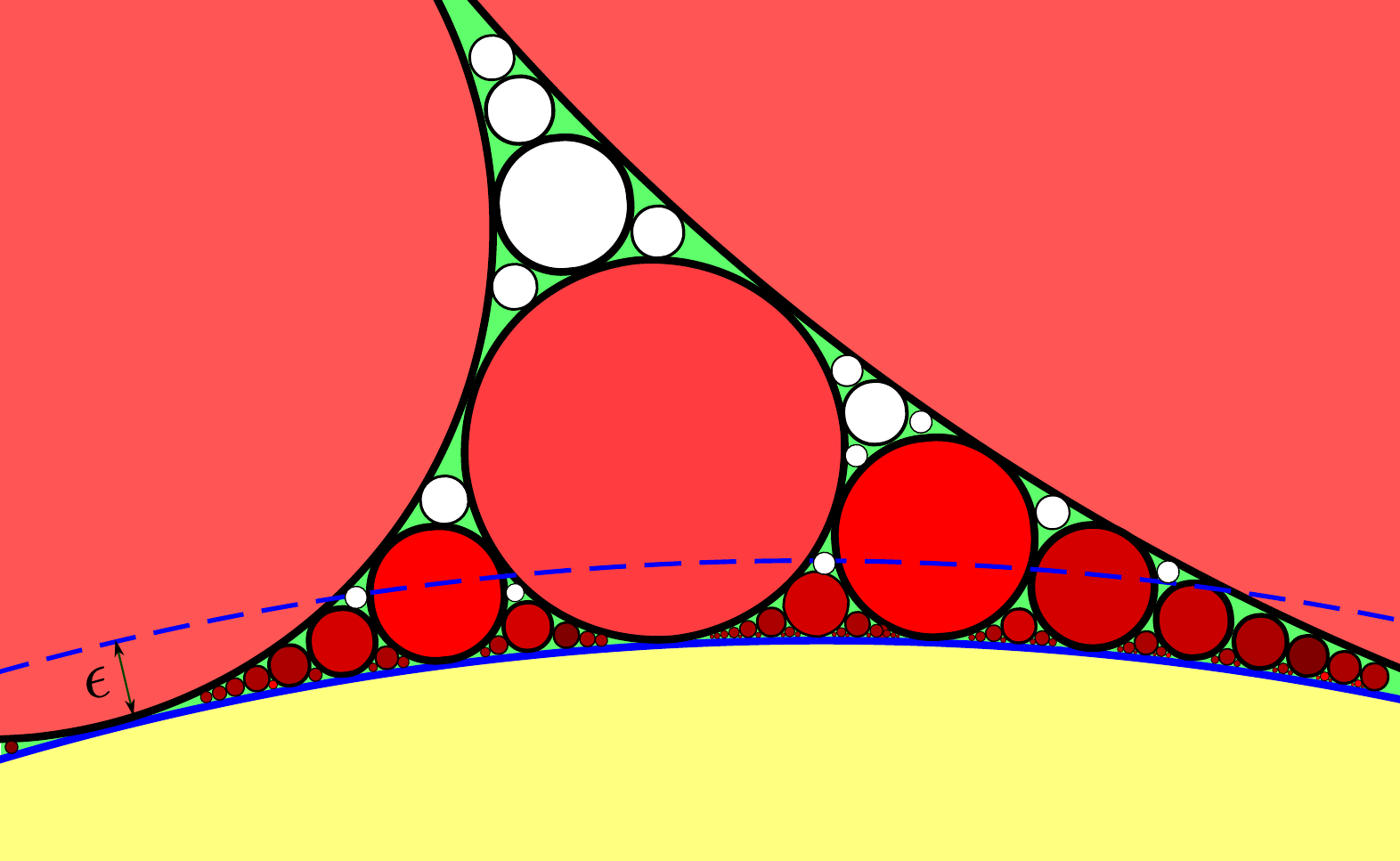}
\vspace*{7mm}
\captionsetup{width=1.0\textwidth}
   \caption{\small 
A portion of a bounded packing. Here $C_0$ is an arc of a 
\textcolor[rgb]{0.7,0.6,0}{yellow} disk.
The circles of ${\mathcal P}_0$  are in various shades of \textcolor{red}{red}, 
$C_0$ is in \textcolor{blue}{blue} and 
$C_\epsilon$ is in dashed \textcolor{blue}{blue}.
}
   \label{fig:config0}
\end{minipage}
\hspace*{\fill}
 \end{figure}

\subsection{Radial Density for Circle Packings}\label{subsec:density packings} 
The main problem we consider is the following: Fix an ACP 
\begin{equation*}
     \p = \bigcup C_i \subset \hat{\C},
\end{equation*}
where $C_i = \partial B(z_i, r_i), z_i \in \C, r_i>0$ or $C_i$ is given by a straight line. We write
$C_i \sim C_j$ if the circles are tangent. Given $C_0 = \partial B(z_0, r_0)$ in $\p$, let 
\begin{equation*}
     \p_0 = \bigcup_{C_i \in \p,\, C \sim C_0} D_i
\end{equation*}
denote the set of disks $D_i = B(z_i, p_i)$ tangent to $C_0$.
Take the concentric circle $C_{\epsilon} = \partial B(z_0, r_0 + \epsilon)$ (in the case of a
bounded packing, if $C_0$ is the outer circle, we set 
$C_{\epsilon} = \partial B(z_0, r_0 -\epsilon)$; 
and if $C_0$ is a line, we let $C_{\epsilon}$ be the parallel line on the `inside' of
the packing, see Figures~\ref{fig:config1}, ~\ref{fig:config2}, ~\ref{fig:config3}). 
A similar setting, where $C_0$ is just an arc of a circle of $\p$, is shown in
Figure~\ref{fig:config0}.

Let
$\mu_{\epsilon}$ denote the Lebesgue probability measure on 
$C_{\epsilon}$ (if $C_{\epsilon} \cong \R$, we take any subset of the packing sitting over a fixed length subinterval of $\R$). We have:

\begin{Theorem}\label{theorem:main:circle} 
Fix notation as above. Then 
\begin{equation*}
     \lim_{\epsilon \rightarrow 0} \mu_{\epsilon} (\p_0) = \frac{3}{\pi}.
\end{equation*}
\end{Theorem}

\medskip

We also consider another family of circles approximating $C_0$. Let $w_0$ denote the point of
tangency between $C_0$ and $C_1$, and for $0< \epsilon < 1$, let $\cc_{\epsilon}$ be the circle containing $C_0$
through $w_0$, tangent to $C_1$, and of radius $r_0 + \epsilon$. Let $\tilde{\mu}_{\epsilon}$ denote
the Lebesgue probability measure on $\cc_{\epsilon}$. Then we have:

\begin{Theorem}\label{theorem:main:horo} 
Fix notation as above. Then 
\begin{equation*}
     \lim_{\epsilon \rightarrow 0} \tilde{\mu}_{\epsilon} (\p_0) = \frac{3}{\pi}.
\end{equation*}

\end{Theorem}

%
%

\medskip

\noindent 
We remark that $\frac{3}{\pi}$ is the \emph{maximal cusp density} for a finite volume
two-dimensional hyperbolic orbifold~\cite{Adams:2002}, equivalently it is the maximal density of a
horoball packing in the hyperbolic plane $\h^2$.  We prove 
Theorems~~\ref{theorem:main:circle} and \ref{theorem:main:horo} 
 in \S\ref{sec:gen} below,
after discussing their proofs for a crucial example in~\S\ref{sec:farey:ford}.

\begin{figure} [tbh]
\hfill
\begin{minipage}[c]{0.48\textwidth}
   \includegraphics[width=0.99\textwidth]{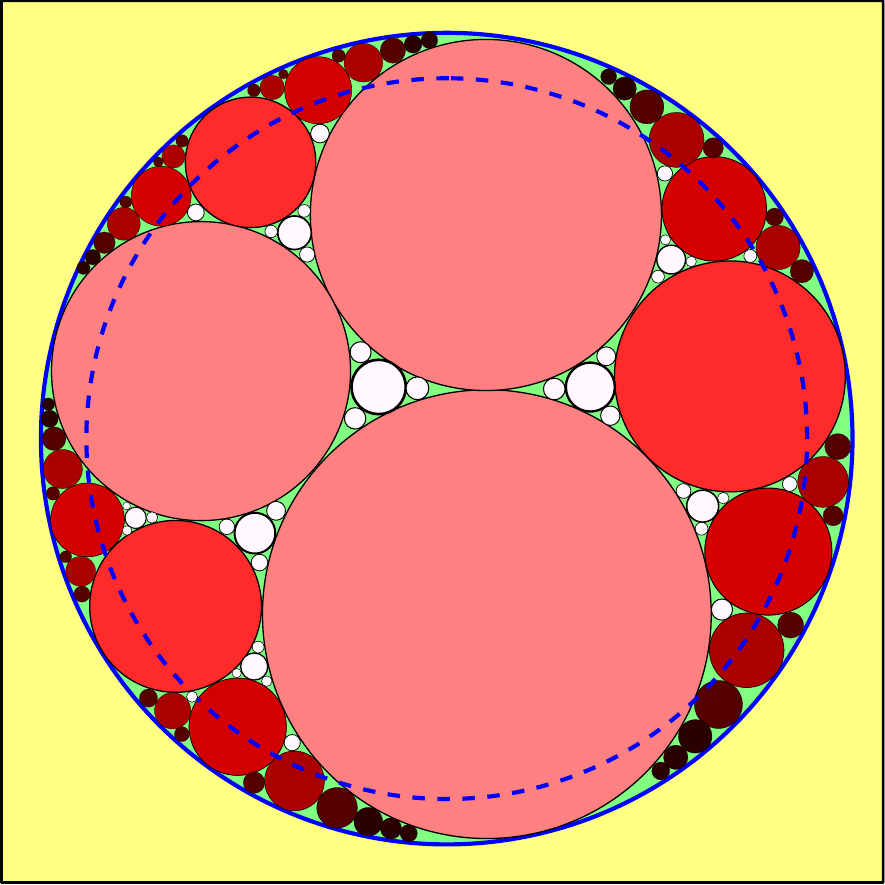}
\captionsetup{width=0.98\textwidth}
   \caption{\small A portion of a bounded packing, where $C_0$ is the outer circle.
The circles of ${\mathcal P}_0$  are in various shades of \textcolor{red}{red}, 
$C_0$ is in \textcolor{blue}{blue} and 
$C_\epsilon$ is in dashed \textcolor{blue}{blue}.}
   \label{fig:config1}
\end{minipage}
\quad
\begin{minipage}[c]{0.48\textwidth}
   \includegraphics[angle=-90,width=0.99\textwidth]{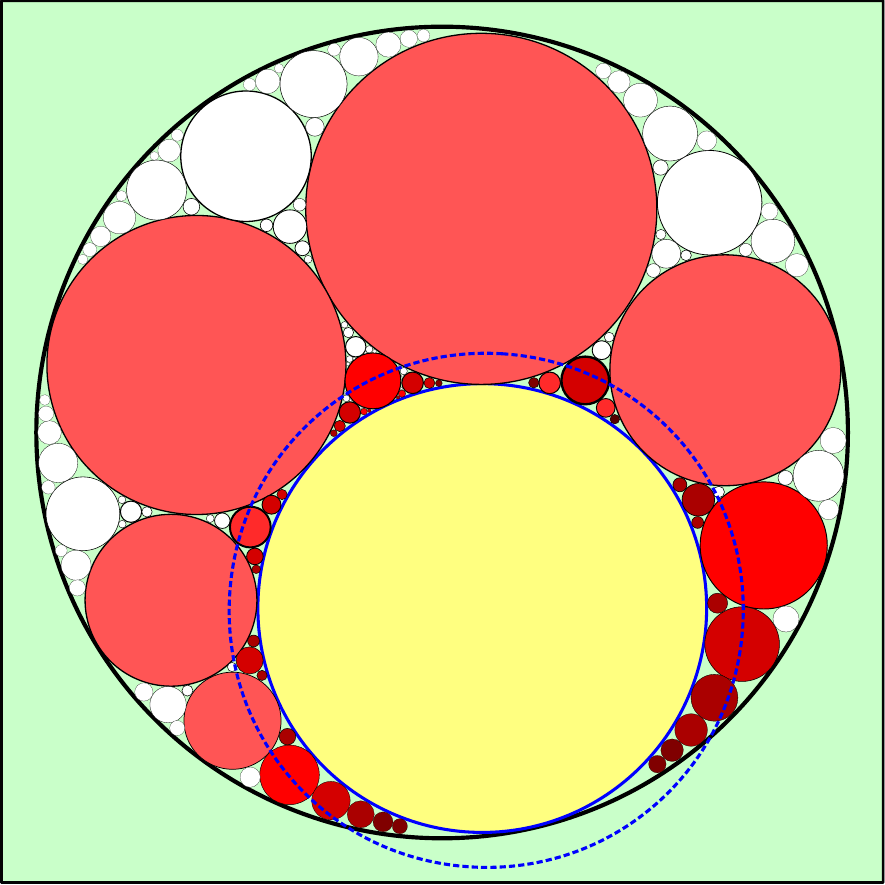}
\captionsetup{width=0.98\textwidth}
   \caption{\small A portion of a bounded packing, where $C_0$ is an inner circle.
The circles of ${\mathcal P}_0$  are in various shades of \textcolor{red}{red}, 
$C_0$ is in \textcolor{blue}{blue} and 
$C_\epsilon$ is in dashed \textcolor{blue}{blue}.
}
   \label{fig:config2}
\end{minipage}
\hspace*{\fill}
 \end{figure}

\subsection{Soddy Sphere Packings}\label{sec:soddy}
In three dimensions, given any $4$ mutually tangent spheres $S_1, S_2, S_3, S_4 \in \R^3 \cup
\infty$, there are two mutually tangent spheres, $S_0$ and $S_5$. Again, by selecting any $4$ of
$S_0, \ldots, S_5$, one obtains new spheres, and iterating this procedure, we obtain a 
\emph{Soddy Sphere packing}, or SSP. As above, consider a SSP 
\begin{equation*}
     \p = \bigcup S_i \subset \hat{\C},
\end{equation*}
where 
$S_i = \partial B(x_i, r_i), x_i \in \R^3, r_i >0$ or $S_i$ is given by a plane. 
We write $S_i \sim S_j$
if the spheres are tangent. Given $S_0 = \partial B(x_0, r_0)$ in $\p$, let 
\begin{equation*}
     \p_0 = \bigcup_{S_i \in \p,\, S \sim S_0} B_i
\end{equation*}
denote the set of balls $B_i = B(x_i, p_i)$ tangent to $S_0$. Take the
concentric sphere $S_{\epsilon} = \partial B(x_0, r_0 + \epsilon)$ (in the case of a bounded
packing, if $S_0$ is the outer sphere, we set $S_{\epsilon} = \partial B(x_0, r_0 - \epsilon)$; and
if $S_0$ is a plane, we let $C_{\epsilon}$ be the parallel plane on the `inside' of the packing. Let
$\nu_{\epsilon}$ denote the Lebesgue probability measure on $S_{\epsilon}$ (if $S_{\epsilon} \cong
\R^2$, we take any subset of the packing sitting over a fixed area subset of $\R^2$). We have:

\begin{Theorem}\label{theorem:main:sphere} 
Fix notation as above. Then 
\begin{equation*}
     \lim_{\epsilon \rightarrow 0} \nu_{\epsilon} (\p_0) = \frac{\sqrt{3}}{2V_T}=0.853\dots,
\end{equation*}
 where $V_T$ denotes the volume of an ideal hyperbolic tetrahedron with dihedral angles $\pi/3$.

\end{Theorem}

\medskip

\noindent 
We also consider another family of spheres approximating $S_0$. Let $y_0$ denote the point of
tangency between $S_0$ and $S_1$, and for $0< \epsilon < 1$, let $\cc_{\epsilon}$ be the sphere containing $S_0$
through $y_0$, tangent to $S_1$, and of radius $r_0 + \epsilon$. Let $\tilde{\nu}_{\epsilon}$ denote
the Lebesgue probability measure on $\cc_{\epsilon}$. Then we have:

\begin{Theorem}\label{theorem:main:horosphere} 
Fix notation as above. Then 
\begin{equation*}
     \lim_{\epsilon \rightarrow 0} \tilde{\nu}_{\epsilon} (\p_0) = \frac{\sqrt{3}}{2V_T}.
\end{equation*}

\end{Theorem}

We remark that $\frac{\sqrt{3}}{2V_T}$ is the \emph{maximal cusp density} 
for a finite volume three-dimensional hyperbolic orbifold~\cite{Adams:2002}, equivalently it is the
maximal density of a horoball packing in the hyperbolic plane $\h^3$. We prove
Theorems~\ref{theorem:main:sphere}  and \ref{theorem:main:horosphere}
in~\S\ref{sec:sphere}.

\section{The Farey-Ford Packing}\label{sec:farey:ford} 
\noindent In this section, we consider the example of the \emph{Farey-Ford Packing}, given
by the horizontal lines $C_0 =\R$, and $C_1 = i + \R$, together with the circles $C_{0, \frac p q},
C_{1, \frac p q}$, for $\frac p q \in \Q$, where $C_{\delta , \frac p q}$ is the circle based at
$\frac p q + i \delta$ of diameter $\frac{1}{q^2}$, for $\delta = 0, 1$. 
The circles $C_{0, \frac p q}$ are known as \emph{Ford circles}, and
two circles $C_{0, \frac p q}$ and $C_{0, \frac r s}$ are tangent if and only if $\frac p q$ and
$\frac r s$ are neighbors in some Farey sequence. 
The packing formed by the Ford circles is illustrated in Figure~\ref{fig:config3}. 

Since this packing is invariant under the
translation $z \mapsto z+1$, we focus on just one period, namely that located between the vertical
lines $x =0$ and $x=1$. 
We set $C_0 = [0, 1]$, $C_{\epsilon} = [0, 1] + i \epsilon$. The collection $\p_0$ of circles
tangent to $C_0$ is given by 
\begin{equation*}
     \F_0 = \displaystyle\bigcup_{\frac p q \in \Q \cap [0, 1]} C_{0, \frac p q},
\end{equation*}
that is, the collection of \emph{Ford Circles} based in $[0, 1)$, and $C_{\epsilon}$ is the
segment $[0,1] + \epsilon i$. Thus, the proportion of $C_{\epsilon}$ contained in $\p_0$ is given
by 
\begin{equation}\label{eq:equi:modular} 
\int_{0}^1 \chi_{\F_0}(x+\epsilon i) \,dx = \sum_{\frac p q : \, q^{-2} > \epsilon} 2\sqrt{\epsilon(
q^{-2} - \epsilon)}\,,
\end{equation}
where each term in the summation is the length of the intersection of the segment 
$C_{\epsilon}$ with the circle $C_{0, \frac p q}$. Below, we give two proofs, one number theoretic
(\S\ref{sec:number}) and one geometric (\S\ref{sec:modular}), that this integral/sum tends to
$\frac{3}{\pi}$ as $\epsilon \rightarrow 0$. 
Moreover, the results  stated in Theorem~\ref{Theorem2} and Theorem~\ref{strom} below further show
that this is an intrinsic property of the packing, because the same limit is attained when $C_0$
and $C_{\epsilon}$ are replaced by any fixed corresponding short intervals.

%
%
%
%
%
%
%
%
%
%

\begin{figure}[htb]
   \centering
   \includegraphics[angle=-90, width=0.87\textwidth]{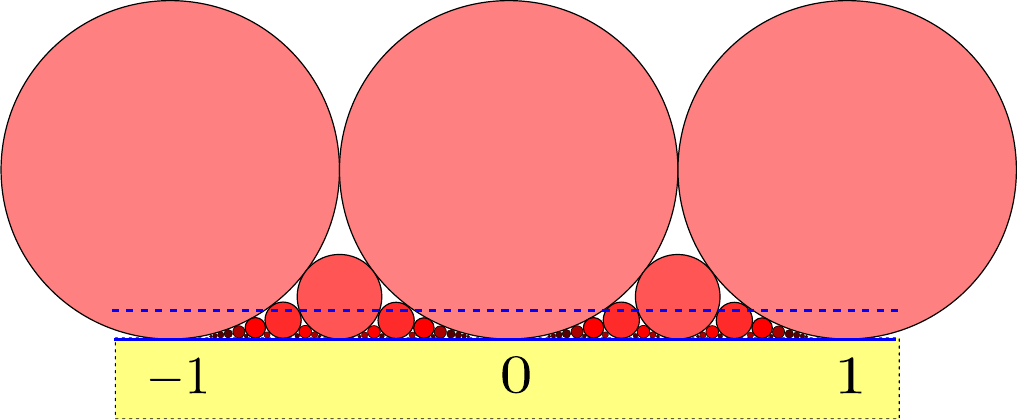}
   \caption{\small An example of a periodic packing is the Farey-Ford Packing $\mathcal F$.
Here, if we take $C_0$ to be the segment $[0,1]$, $\F_0$ consists of Ford Circles, 
 based at $p/q \in [0, 1]$,   diameter $1/q^2$, and $C_{\epsilon}$ is the segment 
 $[0, 1] + \epsilon i$.
}
   \label{fig:config3}
\end{figure}

\subsection{Number Theory}\label{sec:number}
For simplicity of notation, when $\p_0$ is  the Farey-Ford circles and  $C_0$ is the segment
$[0,1]$, we write $L(\epsilon)=\mu_{\epsilon}(\p_0)$.
Then $L(\epsilon)$ is the probability to be in a Ford circle above $C_0$ at height
$\epsilon$.
Our problem
is to understand the behavior of $L(\eps)$ as  $\eps$ decreases from $1/2$ to $0$. 
The probability $L(\eps)$ is an oscillatory function: thus $L(1/2)=1$,
then it decreases till $L(1/4)=0.86602\dots$, than it increases to $L(1/5)=1$, then it
decreases to  $0.87839\dots$, and so on. 
It attains the maximum value $1$ only twice, for $\eps=1/2$ and its  minimum absolute
value is $0.87839\dots$ (see Figures~\ref{fig:Ldeh005} and \ref{fig:Ldeh00002}).
The following theorem shows that there is a limit of $L(\eps)$ as $\eps$ descends to  $0$, and we
also find a bound for the deviation from this limit.

\begin{Theorem}\label{Theorem1}
  For $0<\eps\le 1/2$, we have
\begin{equation}\label{eq3pepiT1}
   L(\eps)=\frac{3}{\pi}+O\left(\sqrt{\eps}\, |\log \eps|\right).
\end{equation}

\end{Theorem}
\begin{proof}
The path in the definition of $L(\eps)$ is a horizontal line situated above the $x$-axis at height
$\eps$ intersects the the Ford circle $C$ of center $(a/q, 1/2q^2 )$ and radius $r=1/2q^2$ if and
only if $\eps\le 1/ q^2$. The length
of the path inside $C$ equals $2\sqrt{r^2 - (\eps-r)^2}$
in both cases $\eps\in(0,r]$ and $\eps\in(r,2r]$, respectively.
Also, because the half paths in the boundary circles (those tangent at $a/q=0$ and $1$) add, we have
\begin{equation}\label{eqLh}
	L(\eps) = \sum_{\substack{1\le a\le q\\ \gcd(a,q)=1 \\ \eps\le2r=1/q^2}} 2\sqrt{r^2 -
(\eps-r)^2} \,.
\end{equation}
The sums can be written as
\begin{equation*}
L(\eps) = \sum_{1\le q\le1/\sqrt{\eps}}\sum_{\substack{1\le a\le q\\(a,q) = 1}}
2\sqrt{2\eps r-\eps^2}
= 2\sqrt{\eps} \sum_{1\le q\le1/\sqrt{\eps }}\varphi(q)\sqrt{\frac1{q^2} - \eps } \,.
\end{equation*}
Put $Q = \lfloor 1/\sqrt{\eps}\rfloor$. Therefore $\sqrt{\eps} = 
\frac1{Q}\left(1+O\Big(\frac1{Q}\Big)\right)$ and
$\eps = \frac1{Q^2}+O\big(\frac1{Q^3}\big)$,
so $L(\eps)$ can be written as
\begin{equation*}
  L(\eps)  = \frac2{Q}\sum_{1\le q\le Q}\varphi(q)\sqrt{\frac1{q^2} - \frac1{Q^2}}
+  O\left(\frac{\log Q}{Q}\right) \,. 
\end{equation*}
Next, using the identity
$\varphi(q) = \sum_{d\mid q}\mu(d)\frac{q}{d}$, we find that
\begin{equation*}
   L(\eps) = \frac2{Q}\sum_{1\le q\le Q}\sum_{d\mid q}\frac{\mu(d)}{d}
\sqrt{1 - \frac{q^2}{Q^2}}
+ O\left(\frac{\log Q}{Q}\right) \,.
\end{equation*}
Here we interchange the order of summation and then break the
sum over $d$ in two sums, according as to whether $d\le D$
or $d> D$, where $D$ is a parameter whose precise value will be chosen later. Thus
\begin{equation}\label{eqS1S2}
   L(\eps)  = S_1(\eps) + S_2(\eps) +  O\left(\frac{\log Q}{Q}\right) \,,
\end{equation}
where
\begin{equation*}
S_1(\eps) = \frac2{Q}\sum_{1\le d\le D} \frac{\mu(d)}{d}
\sum_{1\le m\le \lfloor\frac{Q}{d}\rfloor}
\sqrt{1 - \frac{d^2m^2}{Q^2}}\,,
\end{equation*}
and
\begin{equation*}
S_2(\eps) = \frac2{Q}\sum_{D < d\le Q} \frac{\mu(d)}{d}
\sum_{1\le m\le \lfloor\frac{Q}{d}\rfloor}
\sqrt{1 - \frac{d^2m^2}{Q^2}}\,.
\end{equation*}
The second sum is bounded by
\begin{equation}\label{eqS2}
   |S_2(\eps)| \ll \frac1{Q}\sum_{D < d\le Q} \frac{Q}{d^2} = O\left(\frac1{L}\right).
\end{equation}

In the first sum $d$ is small, $\lfloor\frac{Q}{d}\rfloor$ is large, and the inner
sum is close to the integral $\int_0^{Q/d}\sqrt{1-\frac{d^2x^2}{Q^2}}\, dx$.
Precisely, since the function $x\mapsto\sqrt{1-\frac{d^2x^2}{Q^2}}$ is decreasing on
the interval $[0, \frac{Q}{d}]$, it follows that
\begin{equation*}
\sum_{m=0}^{\lfloor\frac{Q}{d}\rfloor}\sqrt{1 - \frac{d^2m^2}{Q^2}} >
\int_0^{Q/d}\sqrt{1-\frac{d^2x^2}{Q^2}}\, dx >
\sum_{m=1}^{\lfloor\frac{Q}{d}\rfloor}\sqrt{1 - \frac{d^2m^2}{Q^2}}\,.
\end{equation*}
Therefore,
\begin{equation*}
S_1(\eps) = \frac2{Q}\sum_{1\le d\le D} \frac{\mu(d)}{d}
\left(\int_0^{Q/d}\sqrt{1-\frac{d^2x^2}{Q^2}}\, dx + O(1) \right).
\end{equation*}
Next, by a change of variable, 
\begin{equation*}
\int_0^{Q/d}\sqrt{1-\frac{d^2x^2}{Q^2}}\, dx
=\frac Qd \int_0^1\sqrt{1-t^2}\, dt\,.
\end{equation*}
One has $ \int_0^1\sqrt{1-t^2}\, dt = \frac{\pi}{4}$\,, and so
\begin{equation*}
S_1(\eps) = \frac2{Q}\sum_{1\le d\le D} \frac{\mu(d)}{d}
\left(\frac{\pi Q}{4d}+ O(1) \right)
= \frac{\pi}{2}\sum_{1\le d\le D} \frac{\mu(d)}{d^2}
+ O\left(\frac{\log D}{Q}\right).
\end{equation*}
Since
\begin{equation*}
\sum_{1\le d\le D} \frac{\mu(d)}{d^2} = \sum_{d=1}^{\infty} \frac{\mu(d)}{d^2} + 
 O\left(\frac{1}{D}\right)
 = \frac6{\pi^2} + O\left(\frac{1}{D}\right),
\end{equation*}
it follows that
\begin{equation}\label{eqS1}
   S_1(\eps) = \frac3{\pi} + O\left(\frac{1}{D}\right) + O\left(\frac{\log D}{Q}\right).
\end{equation}

Combining the estimates \eqref{eqS2} and \eqref{eqS1} and balancing the error terms by taking
$D=Q/\log Q$, relation \eqref{eqS1S2} becomes
\begin{equation}\label{eqS3}
   L(\eps)= \frac3{\pi} +  O\left(\frac{\log Q}{Q}\right),
\end{equation}
and \eqref{eq3pepiT1} follows, which concludes the proof of the theorem.
\end{proof}

\begin{figure} [tbh]
\begin{minipage}[c]{0.49\textwidth}
   \includegraphics[angle=-90,width=0.99\textwidth,totalheight=0.59\textwidth]{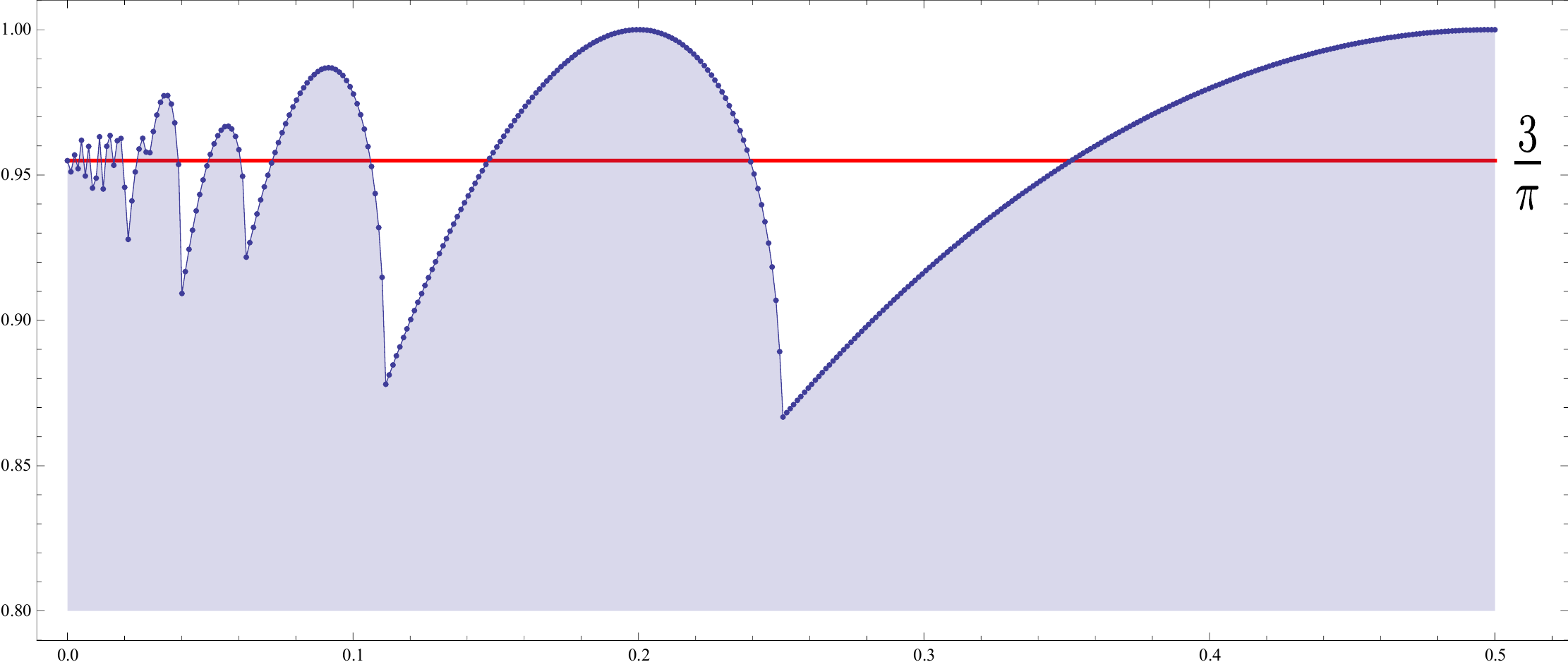}
\captionsetup{width=1.01\textwidth}
   \caption{\small The graph of $L(\eps)$, $0< \eps\le 0.57$.}
   \label{fig:Ldeh005}
\end{minipage}
\ 
\begin{minipage}[c]{0.49\textwidth}
  
\includegraphics[angle=-90,width=0.99\textwidth,totalheight=0.59\textwidth]{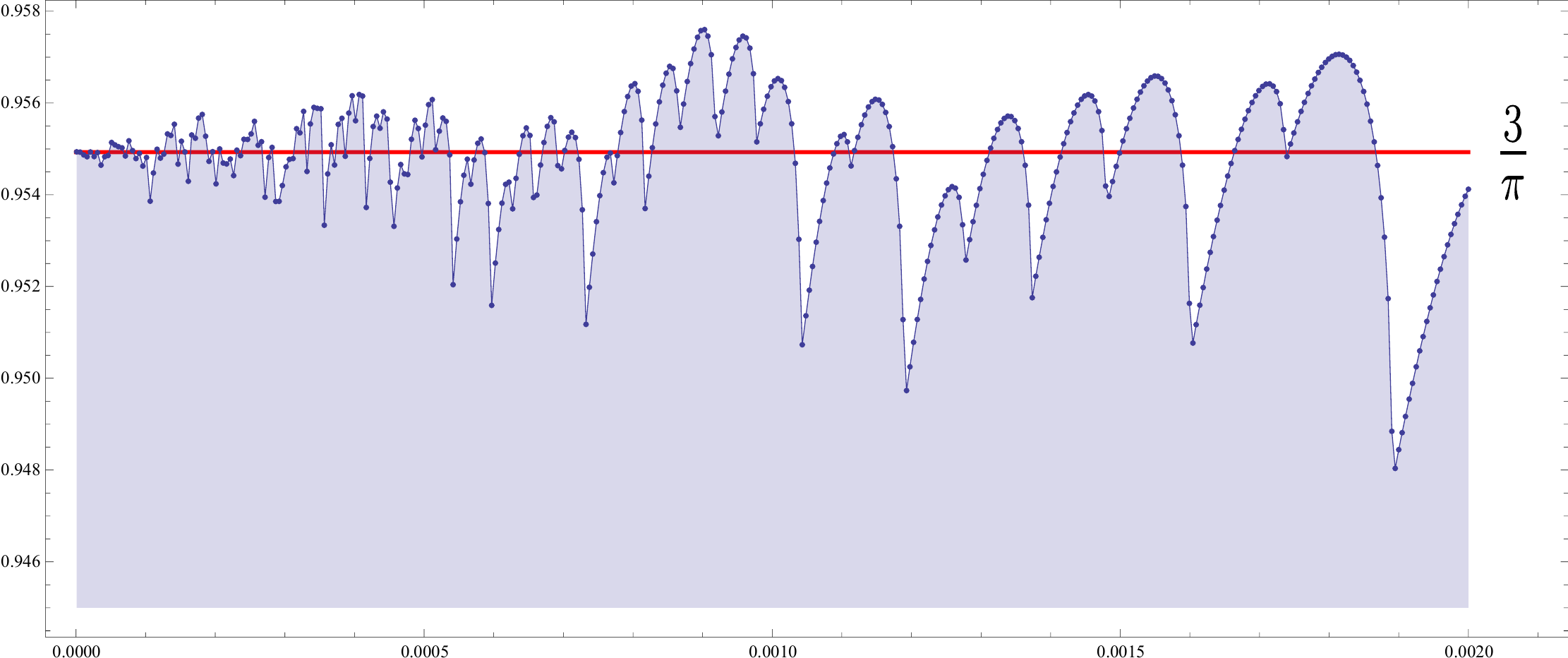}
\captionsetup{width=1.03\textwidth}
   \caption{\small The graph of $L(\eps)$,  $0< \eps\le 0.002$.}
   \label{fig:Ldeh00002}
\end{minipage}
 \end{figure}

As one can see in Figures~\ref{fig:Ldeh005} and \ref{fig:Ldeh00002},
 $L(\eps)$ is a highly oscillatory function, and this behavior accentuates as $\eps\downarrow 0$.
\subsubsection{Short intervals}
Let $\I$ be a subinterval of $[0,1]$. 
Let $L_{\I}(\eps)$ be the probability that a random point at height $\eps$ lying above the interval
$\I$ falls inside of a Ford circle.
In other words, in the general notation $L_{\I}(\eps)=\mu_{\eps}(\p_0(\I))$, where $\p_0(\I)$ is
the collection of Ford circles that are tangent to the $x$-axis at a point from $\I$.
The following theorem shows that in this case also, the limit probability is the distinguished
value $3/\pi$.

\begin{Theorem}\label{Theorem2}
Let $\I$ be a subinterval of $[0,1]$ and let $\delta>0$.
  Then, for $0<\eps\le 1/2$, 
\begin{equation*}\label{eq3pepiT2}
   L_{\I}(\eps)=\frac{3}{\pi}+O_\delta\left(|\I|^{-1}\eps^{1/2-\delta}\right).
\end{equation*}

\end{Theorem}

\begin{proof}

We start as above.                                           
We have
\begin{equation*}
	L_{\I}(\eps)= \frac{S_{\I}(\eps) }{|\I|},
\end{equation*}
where
\begin{equation*}
	S_{\I}(\eps)  = \sum_{\substack{1\le a\le q\\ a/q\in\I \\ \eps\le2r=1/q^2}} 2\sqrt{r^2 -
(\eps-r)^2} .
\end{equation*}
Next,
\begin{equation*}
\begin{split}
   S_{\I}(\eps)  &= \sum_{1\le q\le1/\sqrt{\eps}}\sum_{\substack{a\in q\I \\(a,q) = 1}}
2\sqrt{2\eps r-\eps^2}\\
&= 2\sqrt{\eps} \sum_{1\le q\le1/\sqrt{\eps}}\Big(|\I|\varphi(q)+O\big(\tau(q)\big)\Big)
\sqrt{\frac1{q^2}- \eps}\\
&= |\I|\cdot L(\eps)+O_\delta\left(\sqrt{\eps}
\sum_{1\le q\le1/\sqrt{\eps}}q^{\delta-1}
\right)\\
&= |\I|\cdot L(\eps)+O_\delta\left(\eps^{1/2-\delta}\right),
\end{split}
\end{equation*}
where $\tau(q)$ stands for the number of divisors of $q$.
Thus
\begin{equation*}
	L_{\I}(\eps)=L(\eps)+O_\delta\left(|\I|^{-1}\eps^{1/2-\delta}\right),
\end{equation*}
which, in view of Theorem~\ref{Theorem1}, concludes the proof of the theorem.

\end{proof}

\subsection{Equidistribution on the Modular Surface}\label{sec:modular} 
To prove equation (\ref{eq:equi:modular}), we use Zagiers's theorem~\cite[page 279]{Zagier} on
equidistribution of closed horocycles on the modular surface $\h^2/SL(2, \Z)$ (this was generalized by Sarnak~\cite{Sarnak:1981} to general finite-volume hyperbolic surfaces). In our context, we
can state this theorem as follows:
\smallskip
\begin{Theorem}\label{theorem:sarnak}
Let $f \in C_c(\h^2/SL(2, \Z))$. Lift $f$ to an $SL(2, \Z)$-invariant function (also denoted $f$) on
$\h^2$. Let $\mathcal F \subset \h^2$ denote the $SL(2, \Z)$-fundamental domain 
\begin{equation*}
     \mathcal F = \{z = x+iy \in \h^2: |z| > 1, |x| < 1/2\}.
\end{equation*}
Then 
\begin{equation} 
\lim_{\epsilon \rightarrow 0} \int_{0}^1 f(x+\epsilon i)\, dx 
= \frac{3}{\pi} \int_{\mathcal F} f(x,y) \,\frac{dx dy}{y^2}.
\end{equation}\end{Theorem}

\medskip
Relation \noindent (\ref{eq:equi:modular}) now can be derived as follows: consider the subset
$\Om\subset \mathcal F$ given by 
\begin{equation*}
     \Om  = \{z = x+iy \in \h^2: y > 1, |x| < 1/2\}.
\end{equation*}
 Note that
\begin{equation*}
     \int_{\Om} \frac{dx\,dy}{y^2} = 1 \quad \text{and}\quad
\int_{\mathcal F} \frac{dx\,dy}{y^2} =\frac{\pi}{3}.
\end{equation*}
Denote the union of the interiors of the family
of circles $\F_0$  by $\bar{\F_0}$. This given by the orbit of $\Om$ under the non-upper triangular matrices in $SL(2, \Z)$, that is,  
$$     \bar{\F_0} = \bigcup_{\gamma \in  \in SL(2, \Z), c(\gamma) \neq 0} \gamma \Om,$$ where we write $$\gamma =  \begin{pmatrix}a(\gamma) & b(\gamma) \\c(\gamma) & d(\gamma)\end{pmatrix}$$
Let $g$ be the $SL(2,\Z)$-invariant function whose restriction to the fundamental domain $\mathcal F$ is $\chi_{\Omega}$. Then, $f=1-g$ is compactly supported viewed as a function on $\h^2/SL(2,\Z)$. By construction $f$ and
$g$ are $SL(2, \Z)$-invariant functions on $\h^2$, and, as above, we use the corresponding letters
to
denote the associated functions on $\h^2/SL(2, \Z)$. Then $f \in C_c(\h^2/SL(2,\Z))$, so applying
Theorem~\ref{theorem:sarnak}, we obtain 
\begin{equation*}
     \lim_{\epsilon \rightarrow 0} \int_{0}^1 f(x+\epsilon i)\,dx 
= \frac{3}{\pi} \int_{\mathcal F} f(x,y) \frac{dx \, dy}{y^2} 
= \frac{3}{\pi}  \int_{\mathcal F \backslash \Omega} \frac{dx \,dy}{y^2} = 1- \frac{3}{\pi}.
\end{equation*}
Thus 
\begin{equation*}
    \lim_{\epsilon \rightarrow 0} \int_{0}^1 \chi_{\F_0}(x+\epsilon i)\, dx 
     = \lim_{\epsilon \rightarrow 0} \int_{0}^1 g(x+\epsilon i) \,dx 
     = \lim_{\epsilon \rightarrow 0} \left(1 - \int_{0}^1 f(x+\epsilon i) \,dx \right) 
     = \frac{3}{\pi}. 
\end{equation*}

\qed\medskip

\subsection{Kontorovich's observation on the connection to the Riemann Hypothesis}\label{riemann} In this subsection, we reproduce the contents of a letter written to us by Alex Kontorovich~\cite{Kontorovich:2014}, on how to express the Riemann hypothesis in terms of the convergence of $L(\epsilon)$. We first note that $L(\epsilon)$ can be expressed explicitly in terms of the Riemann zeta function. Precisely, if we define $\vf(s)$ by
$$
\vf(s):=\sqrt{\pi}{\Gamma(s-\foh)\over\Gamma(s)}{\gz(2s-1)\over\gz(2s)},
$$
we have
\begin{equation}\label{eq:Lh2}
L(\epsilon) =
\frac 3\pi+
\frac{\sqrt \epsilon}{4\pi}\int_{\R}
\left({1\over \foh+it}+{\vf(\foh-it)\over \foh-it}\right)
\left(
\epsilon^{it}+\vf(\tfrac12+it)\epsilon^{-it}
\right)
dt
.
\end{equation}

\

\noindent To see this, note that our $SL(2, \Z)$-invariant function $f$ (with notation as above) can be written as 
$$
f(z):=\sum_{\gamma \in\Gamma_{\infty}\backslash\Gamma }\bo_{\{\Im(\gamma z)\ge 1\}},
$$
where $\Gamma=SL(2,\Z)$ and $\Gamma_{\infty}=\mattwos 1\Z{}1$, and
$$
L(\epsilon)=
\int_{0}^{1}f(x+i\epsilon)dx
.
$$
By the spectral decomposition of automorphic forms (see \cite[Thm 15.5]{IwaniecKowalski}), we have:
$$
f(z)=
{\<f,1\> \over  \<1,1\>} + \sum_{j}
\<f,\vf_{j}\>\vf_{j}(z)
+\frac1{4\pi}\int_{\R}
\<f,E(\tfrac12+it,*)\>E(\tfrac12+it,z)dt
,
$$
where $\vf_{j}$ is an orthonormal basis of Maass cusp forms and $E(s,z)$ is the Eisenstein series.
What we really want is not $f$ but its integral over a horocycle of height $\epsilon$; thus
\beann
L(\epsilon)
&=&
\int_{0}^{1}f(x+i\epsilon)dx
\\
&=&
{\<f,1\>\over \<1,1\>} + \sum_{j}
\<f,\vf_{j}\>\int_{0}^{1}\vf_{j}(x+i\epsilon)dx
\\
&&
\hskip.5in
+\frac1{4\pi}\int_{\R}
\<f,E(\tfrac12+it,*)\>\int_{0}^{1}E(\tfrac12+it,x+i\epsilon)dx\, dt
.
\eeann

Because the $\vf_{j}$ are cusp forms, their contribution vanishes. We have $\<f,1\>=1$ and $\<1,1\>=\vol=\pi/3$; thus the main term is determined. The last term is the constant Fourier coefficient of the Eisenstein series, 
which is (see \cite[(15.13)]{IwaniecKowalski}):
$$
\int_{0}^{1}
E(s,x+i\epsilon)
dx
=
\epsilon^{s}+\vf(s)\epsilon^{1-s}
.
$$
Finally, unfolding the inner product gives
$$
\<f,E(s,*)\>
=
\int_{\Gamma_{\infty}\backslash \mathbb{H}^2}\bo_{\{\Im z\ge1\}}
\overline{E(s,z)}
dz
=
\int_{1}^{\infty}
\int_{0}^{1}
\overline{E(s,x+iy)}
dx
{dy\over y^{2}}
.
$$
Again using the Fourier expansion of the Eisenstein series, we obtain
$$
\<f,E(s,*)\>
=
\int_{1}^{\infty}
(\overline{y^{s}+\vf(s)y^{1-s}})
{dy\over y^{2}}
=
{1\over 1-\bar s}+\vf(\bar s){1\over \bar s}
.
$$
Putting everything together, we obtain
$$
L(\epsilon)
=
\frac 3\pi+
\frac1{4\pi}\int_{\R}
\left({1\over 1/2+it}+\vf(1/2-it){1\over 1/2-it}\right)
\left(
\epsilon^{1/2+it}+\vf(1/2+it)\epsilon^{1/2-it}
\right)
dt
,
$$
as claimed.
\\

\noindent The Prime Number Theorem then implies $$L(\epsilon)=3/\pi + o(\sqrt \epsilon).$$ Alternatively, one can observe that $f(z)$ is itself an Eisenstein-like series, whence by taking a Mellin transform/inverse, and shifting contours further, one sees (as in Zagier \cite{Zagier} and Sarnak \cite{Sarnak:1981}) that $$L(\epsilon)=3/\pi + O(\epsilon^{3/4-\delta})$$ if and only if the Riemann Hypothesis holds. 

\subsection{Shrinking intervals}\label{shrinking} 
Str\"ombergsson~\cite{strom}, building on work of Hejhal~\cite{Hejhal}, strengthened the Sarnak-Zagier theorem to deal with \emph{shrinking intervals}.
In our setting, he showed:

\begin{Theorem}[Str\"ombergsson~\cite{strom}, Theorem 1]\label{theorem:strom}
Let $f \in C_c(\h^2/SL(2, \Z))$. Lift $f$ to an $SL(2, \Z)$-invariant function (also denoted $f$) on
$\h^2$. Let $\mathcal F \subset \h^2$ denote the $SL(2, \Z)$-fundamental domain 
\begin{equation*}
     \mathcal F = \{z = x+iy \in \h^2: |z| > 1, \ |x| < 1/2\}.
\end{equation*}
Let $\delta >0$, and let $\beta(\epsilon) > \alpha(\epsilon)$
be functions of $\epsilon$ so that 
$\beta(\epsilon)-\alpha(\epsilon) > \epsilon^{\frac 1 2 -\delta}$. 
Then  
\begin{equation*} 
\lim_{\epsilon \rightarrow 0}
\int_{\alpha(\epsilon)}^{\beta(\epsilon)} f(x+\epsilon i)\, dx 
= \frac{3}{\pi} \int_{\mathcal F} f(x,y) \frac{dx\, dy}{y^2}.
\end{equation*}
\end{Theorem}

This immediately yields:

\begin{Theorem}\label{strom} 
Let $\delta >0$, and let $\beta(\epsilon) > \alpha(\epsilon)$ be functions of $\epsilon$ so that
$\beta(\epsilon)-\alpha(\epsilon) > \epsilon^{\frac 1 2 - \delta}$. Then  
\begin{equation*} 
     \lim_{\epsilon \rightarrow 0} 
     \int_{\alpha(\epsilon)}^{\beta(\epsilon)} \chi_{\F_0}(x+\epsilon i)\, dx =
     \frac{3}{\pi}\,.
\end{equation*}
\end{Theorem}

\section{General Packings}\label{sec:gen} 
\noindent  In this section, we give geometric proofs of Theorem~\ref{theorem:main:circle} and
Theorem~\ref{theorem:main:horo}. We first state two preliminary lemmas: first, a standard statement
about transitivity of M\"obius transformations (\S\ref{sec:mobius}) and second (\S\ref{sec:shah}), a
deep theorem of Shah's generalizing \mbox{Sarnak's} equidistribution result
Theorem~\ref{theorem:sarnak}.

\subsection{M\"obius Equivalence}\label{sec:mobius} 
The key fact that allows us to reduce the general packing to the Farey-Ford packing is the fact that
the group of  M\"obius transformations (i.e., $SL(2,\C)$) acts transitively on triples of points on
the Riemann sphere $\widehat{\C}$. This allows us to show:

\begin{lemma}\label{lemma:equiv} 
Let $\p$ be an ACP,  $C_0$ a circle in $\p$, and $\p_0$ the collection of circles in $\p$ tangent
to $C_0$. Then there is a M\"obius transformation $M: \widehat{\C} \rightarrow \widehat{\C}$ such
that
$M(C_0) = \R \subset \widehat{\C}$ and  $M(\p_0)$ is the collection of Ford circles $\F_0$,
together with the horizontal line $\{i+t: t\in \R\}$.

\end{lemma}

\begin{proof} 
Consider $C_0$ and two circles tangent to $C_0$ (and each other) in $\p$, call them $C_1$ and $C_2$,
so that $C_0, C_1, C_2$ form a triple of mutually tangent circles. Let $z_0$ denote the point of
tangency between $C_0$ and $C_1$, $z_1$ the point of tangency between $C_0$ and $C_2$, and $z_2$ the
point of tangency between $C_1$ and $C_2$.  Then there is a M\"obius transformation $M$ that
satisfies 
\begin{equation*} 
     f(z_0) = \infty, f(z_1) =0, f(z_2) = i.
\end{equation*}

Since M\"obius transformations preserve
tangency, it follows that $M(C_0)$ is the real axis, $M(C_1)$ is the line $\{ i + t: t \in \R\}$,
 $C_2$ is the circle of radius $\frac{1}{2}$ centered at $\frac{i}{2}$.
Moreover, $M$ maps the packing
$\p$ to the Farey-Ford packing $\F$ and, particularly,
 $M$ maps circles in $\p$ tangent to $C_0$ to
circles in $\F$ tangent to $f(C_0)$, namely, the Ford circles together with the line 
$\{i +t : t\in \R\}$.
\end{proof}

\subsection{Shah's Equidistribution Theorem}\label{sec:shah} 
The main ingredients we require in our general proofs (both here and in \S\ref{sec:sphere}) are
special cases of a very deep result of Shah~\cite[Theorem 1.4]{Shah:1996} on equidistribution of closed
horospheres. In the context of the Farey-Ford packing, this theorem implies the following: let
$\eta$ denote an absolutely continuous (with respect to Lebesgue measure) probability measure on
$\R$. Fix notation as in Theorem~\ref{theorem:sarnak}, in particular, given 
$f \in C_c(\h^2/SL(2, \Z))$, we lift $f$ to an $SL(2, \Z)$-invariant function (also denoted $f$) on
$\h^2$, and $\mathcal F = \{z = x+iy \in \h^2: |z| > 1, |x| < 1/2\}$ is a fundamental domain for the
$SL(2, \Z)$ action on $\h^2$. 
\begin{Theorem}\label{theorem:shah} 
For any $f \in C_c(\h^2/SL(2, \Z))$, we have

 \begin{equation*} 
\lim_{\epsilon \rightarrow 0} \int_{0}^1 f(x+\epsilon i) \, d\eta(x) = \frac{3}{\pi} 
     \int_{\mathcal F} f(x,y)\, \frac{dx\, dy}{y^2},
\end{equation*}

\end{Theorem}
\bigskip

\subsection{Proof of Theorem~\ref{theorem:main:horo}}\label{subsec:horo:proof} 
To prove Theorem~\ref{theorem:main:horo}, note that if $M$ is the transformation as in
Lemma~\ref{lemma:equiv}, the measure $M_* \tilde{\mu}_{\epsilon}$ is supported on the line 
$\{x + \epsilon i: x \in \R\}$, and we can write 
$dM_* \tilde{\mu}_{\epsilon}(x+\epsilon i) = d\eta(x)$, where $\eta$ is an absolutely continuous
probability measure on $\R$. 
Now, applying Theorem~\ref{theorem:shah}, and following the argument at
the end of \S\ref{sec:modular}, we have Theorem~\ref{theorem:main:horo}.
\qed\medskip

\subsection{Eskin-McMullen Equidistribution Theorem} 
Given a general ACP packing $\p$, and $M$ as in
Lemma~\ref{lemma:equiv}, the image $M(C_{\epsilon})$ under the M\"obius transformation $M$ is a
large circle $\partial B(z_{\epsilon}, R_{\epsilon})$ ($R_{\epsilon} \rightarrow \infty$ as
$\epsilon \rightarrow 0$) in the upper half-plane, and the measure $\eta_{\epsilon} =
M_{*}\mu_{\epsilon}$ is supported on this large circle. We write 
 \begin{equation*} 
      M(C_{\epsilon}) = \{z_{\epsilon}, a_{\epsilon} r_{\epsilon, \theta}\}_{\theta \in [0, 2\pi)},
 \end{equation*}
where $z_{\epsilon}$ denotes the
center of the circle, $a_{\epsilon}$ is a fixed diagonal matrix whose trace is 
$\cosh (R_{\epsilon}/2)$, and $r_{\epsilon, \theta}$ denotes rotation by angle $\theta$ around
$z_{\epsilon}$. The measure $\eta_{\epsilon}$ can be then be thought of as a measure on $[0, 2\pi)$,
and is absolutely continuous with respect to Lebesgue measure.

We would like to say that these measures equidistribute when projected to the modular surface as 
$\epsilon \rightarrow 0$, and this is given to us by the following (small modification of a) result
of Eskin-McMullen~\cite[Theorem 1.2]{EsMc:1993}. 

For the sake of clarity, we state it in the special case of hyperbolic manifolds. 
Let $\h^n$ denote $n$-dimensional hyperbolic space, and let $G= SO(n, 1)$. Then $G$ is the isometry
group of $\h^n$, and given a lattice $\Gamma \subset G$, we can form the quotient $X = \h^n/\Gamma$,
a finite-volume hyperbolic orbifold. Let $K = SO(n)$ denote a maximal compact subgroup of $G$, and
$P$ a minimal parabolic subgroup so $G= PK$. Let $A$ be a maximal connected $\R$-diagonalizable
subgroup of $G$ contained in $Z_G(M) \cap P$, where $M = P \cap K$. $A$ is a one-parameter subgroup
of $G$. Given $x_0 \in \h^n$, the sphere of radius $R$ centered at $x_0$ in $\h^n$ can be viewed as
\begin{equation*}
     \{x_0. a_{t} k\}_{k \in K_{0}},
\end{equation*}
where $K_{0}$ is the conjugate of $K$ that stabilizes $x_0$.
Here $t$ is chosen so $d(x_0.a_{t}, x_0) = t$.  Fix an absolutely continuous probability measure
$\eta$ on $K$ (which yields a measure on all the conjugates). Let $\pi: \h^n \rightarrow X$ and let
$dx$ denote the Lebesgue probability measure on $X$.

\begin{Theorem}\label{theorem:eskinmcmullen} 
Let $a_{t_m} \rightarrow \infty$ be a sequence of elements in $A$, $\{x_m\}_{m\ge 0}
\subset \h^n$ such that $x_m \rightarrow y \in \h^n$ and let $K_m$ be the conjugates of $K$
stabilizing $x_m$. Let $\eta_m$ denote the corresponding measures. Then, for any $f \in
C_c(\h^n/\Gamma)$,
 \begin{equation*} 
    \lim_{m \rightarrow \infty} \int_{K_m} f(\pi(x_m. a_m k))\, d\eta_m(k) = \int_{X} f(x)\, dx\,.
 \end{equation*}

\end{Theorem}

\begin{proof} 
The main difference between this result and~\cite[Theorem 1.2]{EsMc:1993} is that the basepoint of the
translates ($x_m$ in our notation) is allowed to vary. We sketch the 
required modification of their proof,
explained to us by N.~Shah~\cite{Shah:1996}. 
Lifting to $G/\Gamma$, our question becomes
whether $a_{t_m}Kg_m\Gamma/\Gamma$ equidistributes as $t_m\to\infty$ and $g_m\to g$. We can first
assume that $g_n\to e$. Next, decompose the translate $Kg_n$ into $U^{-} A U^+$ (where $U^+$ and
$U^-$
are the positive and negative \emph{horospherical} subgroups \mbox{associated to $A$).} Since
$Kg_n$
locally (at essentially all $k\in K$) projects to $U^+$ surjectively, the mixing of the $a_t$ action
gives the desired equidistribution, as in the classical setting.

\end{proof}

\subsection{Proof of Theorem~\ref{theorem:main:circle}}\label{subsec:circle:proof} 
Applying Theorem~\ref{theorem:eskinmcmullen} to the indicator function of the region 
$\{z \in \h^2: |z|>1, \mbox{Im}(z) < 1\}$ as at the end of \S\ref{sec:modular}, we obtain our
result.
\qed\medskip

\section{Sphere Packings}\label{sec:sphere}
 To prove Theorem~\ref{theorem:main:sphere} and Theorem~\ref{theorem:main:horosphere}, we follow a
similar strategy as in the circle packing case.

\subsection{The base packing}
 Again using the transitivity of the action of M\"obius maps, we can map any packing to a base
packing
given by a packing between the two planes $x_3=0$ and 
$x_3=1$ in $\R^3 = \{(x_1, x_2, x_3): x_i \in \R\}$, including spheres tangent to points $(x_1,
x_2, 0)$ and $(x_1, x_2, 1)$ at points $(x_1, x_2)$ such that $x_1 + i x_2  \in \Q[\sqrt{-3}]$. The
view of (a part of) this packing from below the $xy$-plane is shown in Figure~\ref{fig:3dbottom}.
The
stabilizer of this packing, $\Gamma_3$, is a non-uniform lattice in the isometry group of $\h^3$ and
has been explicitly computed (and was named the \emph{Soddy group}) by
Kontorovich~\cite{Kontorovich:2012}. 
$\Gamma_3$ can be expressed as follows: 
 \begin{equation*}
      \Gamma_3 : = \left\{
	  \left(\begin{array}{cc}a & b \\c & d\end{array}\right) \in SL_2(\mathcal O_3): \ 
	       b \equiv 0	  \pmod 3\right\},
 \end{equation*}
where $\mathcal O_3$ is the ring of integers in $\Q[\sqrt{-3}]$. 

Similar to the
circle packing setting, the density we are interested in corresponds to computing how much of a
plane at height $\epsilon$ intersects the spheres which are tangent to the plane $x_3 =0$.  
\begin{figure}
\centering
\includegraphics[height=80mm, width=80mm]{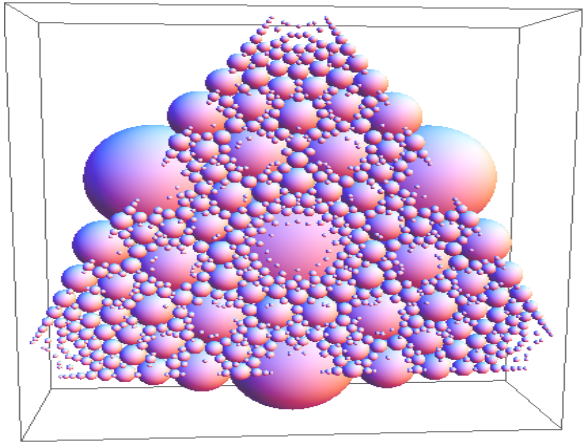}
\caption{\small The base packing viewed from below the $xy$-plane.}\label{fig:3dbottom}
\end{figure}
As $\epsilon \rightarrow 0$, this corresponds to how much of a large closed horosphere 
intersects the cusp in the hyperbolic manifold $\h^3/\Gamma_3$. The cusp density of $\h^3/\Gamma_3$
can be computed explicitly using the fundamental domain computed by
\mbox{Kontorovich~\cite[\S3]{Kontorovich:2012}}. The
fundamental domain is the union of two ideal hyperbolic tetrahedra with dihedral angles $\pi/3$, and the cusp corresponds to the
region above the plane tangent to the bottom spheres of the fundamental domain. The density can then
be seen to be $\frac{\sqrt{3}}{2V_T}=0.853\dots$, where $V_T$ denotes the volume of an ideal
hyperbolic tetrahedron with dihedral angles $\pi/3$. This is in fact the \emph{same} as the cusp density of $\h^3/\Gamma_8$,
where $\h^3/\Gamma_8$ is the figure-8 knot complement. As in the two-dimensional case, this number
is the maximal cusp density of a hyperbolic $3$-orbifold (see Adams~\cite{Adams:2002}), or equivalently,
the maximal
density of a horoball packing of $\h^3$.

\begin{figure}[h]
\centering
\includegraphics[width=0.87\textwidth]{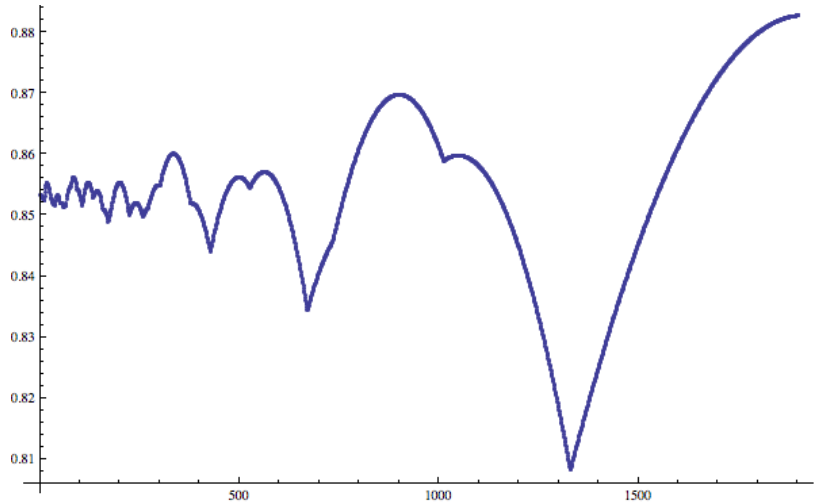}
\captionsetup{width=0.78\textwidth}
\caption{\small The horospherical density in the base packing, computed numerically. The
$x$-axis is the height of the intersecting horosphere and the $y$-axis is the
density.}\label{fig:3dstat}
\end{figure}

\subsection{Horospheres} 
To prove Theorem~\ref{theorem:main:horosphere}, we again convert our given packing to the base
packing, obtaining a new family of measure on the planes of height $\epsilon$. Applying the
$3$-dimensional version of Shah's equidistribution theorem, we obtain the result. For the base
packing, you can see the numerical computations in Figure~\ref{fig:3dstat}. For further details on these numerical experiments, see Hempstead~\cite{Hempstead}.\qed\medskip

\subsection{Spheres} 
To prove Theorem~\ref{theorem:main:sphere}, we apply the (modified) Eskin-McMullen equidistribution
theorem in
this setting. \qed\medskip

\subsubsection{Curves on Spheres} 
We remark that we can also prove a strengthening of our result restricted to appropriate curves on
spheres, using another beautiful equidistribution result of Shah~\cite[Theorem 1.1]{Shah1:2009}, which
allows us to consider the radial density (in the sphere packing setting) restricted to analytic 
curves on our spheres which are not contained in any proper subsphere. Namely, we have the
following: let $\phi: [0, 1] \rightarrow S^2$ be an analytic curve whose image is not contained in
any proper subsphere. Let $\psi_{\epsilon}: S^2 \rightarrow S_{\epsilon}$ denote the natural
parameterization of the concentric spheres $S_{\epsilon}$ (notation as in \S\ref{sec:soddy}), and
let $\eta_{\epsilon} = (\psi_{\epsilon} \circ \phi)_* m$ denote the push-forward of the Lebesgue
probability measure on $[0,1]$ onto $S_{\epsilon}$. Then we have 
\begin{equation*}
     \lim_{\epsilon \rightarrow 0} \eta_{\epsilon} (\p_0) = \frac{\sqrt{3}}{2V_T}=0.853\dots,
\end{equation*}
In fact, this result can even be extended to appropriate families of smooth curves,
following~\cite[Theorem 1.1]{Shah2:2009}.

\medskip
\noindent\textbf{Acknowledgements.} 
We would like to thank Florin Boca, Yair Minsky, and Nimish Shah for useful
discussions. We would like to thank Alex Kontorovich for many helpful conversations and for contributing section~\ref{riemann}. We would like to thank Gergely Harcos for his comments on an early draft of this paper, and the anonymous referee for their suggestions and corrections. Finally, we would like to thank Jason Hempstead, Kaiyue Hou, and Danni Sun of the
Illinois
Geometry Lab for their work on numerical experiments.


\end{document}